\documentclass{birkjour}

\usepackage[utf8]{inputenc}
\usepackage[T1]{fontenc}
\usepackage{amsmath, amsfonts, amssymb, amsthm}
\usepackage[none]{hyphenat}
\usepackage{graphicx}
\usepackage{float}
\usepackage{times}
\usepackage{esint}
\usepackage{lipsum}
\usepackage{enumitem}
\usepackage{amsmath, amsthm, amscd, amsfonts, amssymb, graphicx, color}
\usepackage[bookmarksnumbered, colorlinks, plainpages]{hyperref}
\hypersetup{colorlinks=true,linkcolor=red, anchorcolor=green, citecolor=cyan, urlcolor=red, filecolor=magenta, pdftoolbar=true}
\usepackage{orcidlink}
\usepackage[scr]{rsfso}
\newtheorem{theorem}{Theorem}[section]
\newtheorem{lemma}[theorem]{Lemma}
\newtheorem{proposition}[theorem]{Proposition}

\theoremstyle{definition}
\newtheorem{definition}[theorem]{Definition}

\newtheorem{remark}[theorem]{Remark}
\numberwithin{equation}{section}

\numberwithin{equation}{section}

\begin{document}

%-------------------------------------------------------------------------
% editorial commands: to be inserted by the editorial office
%
%\firstpage{1} \volume{228} \Copyrightyear{2004} \DOI{003-0001}
%\seriesextraline{This is the Concrete Title of this Book\br H.E. R and S.T.C. W, Eds.}
%\firstpage{1}

\title{Some approximation properties in fractional Musielak-Sobolev spaces}

%----------Author 1
\author{Azeddine BAALAL}
\address{{ Department of Mathematics and Computer Science},
	{A\"in Chock Faculty,  Hassan II University}, 
	{ B.P. 5366 Maarif, Casablanca},
	{Morocco}}
\email{abaalal@gmail.com}

%----------Author 2
\author{Mohamed BERGHOUT}
 	\address{{Laboratory of Partial Differential Equations, Algebra and Spectral Geometry, Higher School of Education and Training, Ibn Tofail University}, {P.O. Box 242-Kenitra 14000, Kenitra}, {Morocco}}
\email{Mohamed.berghout@uit.ac.ma; moh.berghout@gmail.com}
 
 %----------Author 3
 \author{EL-Houcine OUALI* \orcidlink{0009-0007-9106-6441}}
\address{{ Department of Mathematics and Computer Science},
	{A\"in Chock Faculty,  Hassan II University}, 
	{ B.P. 5366 Maarif, Casablanca},
	{Morocco}}
\email{oualihoucine4@gmail.com}


\date{}

\begin{abstract}
	  In this article we show some density properties of smooth and compactly supported functions in fractional Musielak-Sobolev spaces essentially extending the results of Fiscella, Servadei and Valdinoci in \cite{fiscella2015density} obtained in the fractional Sobolev setting. The proofs of this properties are mainly based on a basic technique of convolution (which makes functions $C^{\infty}$), joined with a cut-off (which makes their support compact), with some care needed in order not to exceed the original support.
\end{abstract}

%%% ----------------------------------------------------------------------
\maketitle
%%% ----------------------------------------------------------------------

\section{Introduction and main results}
Recently, great attention has been focused on problems involving the theory of fractional modular spaces, in particular the fractional Sobolev spaces with variable exponents $W^{s,q(.),p(.,.)}(\Omega)$ (see \cite{baalal2018traces,babm2018density,barv2018onanew,dlrj2017traces,kjv2017fract,km2023bourgan}) and the fractional order Orlicz- Sobolev spaces $W^{s,G}(\Omega)$ (see \cite{bsoh2020embedding,bot2020basic,bor2023anewclass,BjSa2019}), which are two distinct extensions of the so-called fractional Sobolev spaces $ W^{s,p}(\Omega)$ (see \cite{bms2019radial,di2012hitchhikers}),  and they are two special kinds of fractional Musielak-Sobolev spaces $W^{s,G_{x,y}}(\Omega)$ (see \cite{abss2022class, abss2023emb,bamhoh2024onthefract,aacs2023onfrac}). Particularly,
one of these problems is the density of smooth and compactly supported functions in these spaces.

Our paper is motivated by the article \cite{fiscella2015density}, where the authors consider the Sobolev space $X^{s,p}_{0}(\Omega)$ of functions $f$ with the finite norm
$$
\|f\|_{L^p\left(\mathbb{R}^N\right)}+\left(\int_{\mathbb{R}^N \times \mathbb{R}^N}|f(x)-f(y)|^p K(x-y) d x d y\right)^{1 / p},
$$
but vanishing outside $\Omega$, with some assumptions on the kernel $K$. The authors proved that the space $C_0^{\infty}(\Omega)$ defined by
\begin{equation}\label{Smooth}
C_0^{\infty}(\Omega)=\lbrace g: \mathbb{R}^N \rightarrow \mathbb{R}: g \in C^{\infty}\left(\mathbb{R}^N\right), \textit{ Supp } g \textsl{  is compact and } \textit{Supp } g \subseteq \Omega\rbrace,
\end{equation}
where Supp $g=\overline{\left\{x \in \mathbb{R}^N: g(x) \neq 0\right\}}$, is dense in $X_0^{s, p}(\Omega)$ when $\Omega$ is either a hypograph or a domain with continuous boundary (see \cite{fiscella2015density}, Theorem 2 and Theorem 6 ).

Let us also mention other articles on similar topics. In \cite{baalal2024density} the authors extended the results obtained in \cite{fiscella2015density} to the fractional Orlicz–Sobolev framework, proving that the space $C_0^{\infty}(\Omega)$ is dense in $W_0^{s, G}(\Omega)$ when $\Omega$ is either a hypograph or a domain with continuous boundary, where the space $W_0^{s, G}(\Omega)$ is defined by 
$$W_0^{s, G}(\Omega):=\left\lbrace u\in W^{s, G}(\mathbb{R}^{N}): u=0 \textsf{ a.e. in  } \mathbb{R}^{N}\setminus\Omega \right\rbrace. $$

In \cite{babm2018density}, the authors considered fractional Sobolev spaces with variable exponents $W^{s, q(\cdot), p(\cdot,\cdot)}(\Omega)$ and proved that under certain conditions for the functions $p$ and $q$, the space of smooth and compactly supported functions is dense in $W^{s, q(\cdot), p(\cdot,\cdot)}(\Omega)$.

Our main goal in this paper is to extend the density results obtained in \cite{fiscella2015density} to include the space $W^{s,G_{x,y}}_{0}(\Omega)$ of functions $u\in W^{s,G_{x,y}}(\Omega)$ that vanish a.e outside $\Omega$. Namely \begin{equation}\label{space}
 W^{s,G_{x,y}}_{0}(\Omega):=\left\lbrace u\in W^{s, G_{x,y}}(\mathbb{R}^{N}): u=0 \textsf{ a.e. in  } \mathbb{R}^{N}\setminus\Omega \right\rbrace.
\end{equation} It turns out that for some open subset $\Omega$ of $\mathbb{R}^{N}$ satisfying
certain conditions, smooth and compactly supported functions are dense in $W^{s,G_{x,y}}_{0}(\Omega)$, see Theorem \ref{theo1} and Theorem \ref{theo2}. Our strategy of the proof follows the approach of \cite[Theorem 2 and Theorem 6]{fiscella2015density}, in which we use a basic technique of convolution joined with a cut-off, with some care needed in order not to exceed the original support.

To set our main results, we consider the following definitions and assumptions. Let $N\geqslant 1$, $\Omega$ an open subset in $ \mathbb{R}^{N}$ and  $G: \Omega \times \Omega \times[0, +\infty) \rightarrow$ $[0, +\infty)$ a Carathéodory function defined by
\begin{equation}\label{eq1}
G_{x,y}(t):=G(x, y, t)=\int_0^t g(x, y, \tau) d \tau,
\end{equation}
where
$$g(x, y, t):= \begin{cases}a(x, y, t) t & \text { if } t \neq 0 \\ 0 & \text { if } t=0,\end{cases}$$
with $a: \Omega \times \Omega \times(0, +\infty) \rightarrow[0, +\infty)$ is a function satisfying :
\begin{itemize}
\item[$({g_{1}})$] $\displaystyle\lim _{t \rightarrow 0} a(x, y, t)t=0$ and $\displaystyle\lim _{t \rightarrow +\infty} a(x, y, t)t=+\infty$ for all $(x, y) \in \Omega \times \Omega$;
\item[$(g_{2})$] $t\mapsto a_{x,y}(t):=a(x, y, t)$ is continuous on $(0, +\infty)$ for all $(x, y) \in \Omega \times \Omega$;
\item[$(g_{3})$] $t\mapsto a_{x,y}(t)t$ is increasing on $(0, +\infty)$ for all $(x, y) \in \Omega \times \Omega$.
\item[$(g_{4})$] There exist positive constants $g^{+}$and $g^{-}$ such that
\begin{equation}\label{eq1.5}
1<g^{-} \leqslant \frac{ a_{x, y}(t)t^{2}}{G_{x, y}(t)} \leqslant g^{+}<+\infty \quad \textsl{ for all } (x, y) \in \Omega \times \Omega, \quad  \textsl{ and all  } t > 0 .
\end{equation}
\end{itemize}
Now let us consider the function $\widehat{G}_x: \Omega \times \mathbb{R}^{+}\rightarrow \mathbb{R}^{+} $ given by
\begin{equation}\label{eq2}
\widehat{G}_x(t):=\widehat{G}(x, t):=G(x,x,t)=\int_0^t \widehat{g}(x,\tau) \mathrm{d} \tau .
\end{equation}
where $\widehat{g}(x,t):=\widehat{a}(x,t)t=a(x, x,t)t$ for all $(x, t) \in \Omega \times (0,+\infty)$.
 The assumption $(g_{4})$ implies that
\begin{equation}\label{eq1.6}
1<g^{-} \leqslant \frac{ \widehat{g}(x,t)t}{\widehat{G}(x,t)} \leqslant g^{+}<+\infty, \quad \textsl{ for all } x \in \Omega,\quad \textsl{ and all } t > 0 .
\end{equation}\\
As a consequence of the previous assumptions, one may obtain the following properties:
\begin{itemize}
\item[(i)] $t\rightarrow G_{x,y}(t)$ is continuous,  strictly increasing and convex on $[0,+\infty)$;
\item[(ii)] $\displaystyle\lim _{t \rightarrow 0} \frac{G_{x,y}(t)}{t}=0$;
\item[(iii)]
$
\displaystyle\lim _{t \rightarrow +\infty} \frac{G_{x,y} (t)}{t}=+\infty$;
\item[(iv)] $G_{x,y}(t)>0$ for all $t>0$.
\end{itemize}
See for instance the book of Kufner–John-Fu\v{c}\'{\i}k \cite[Lemma 3.2.2]{kajofs1977}.
\begin{definition}
Let $\Omega$ be an open subset of $\mathbb{R}^{N}$. A function $G: \Omega \times \Omega \times \mathbb{R^{+}} \rightarrow \mathbb{R^{+}}$ is said to be a generalized N-function if it fulfills the properties (i)$-$(iv) above for a.e. $(x, y) \in \Omega \times \Omega$, and for each $t\geq 0$, $G_{x,y}(t)$ is measurable in $(x, y)$.
\end{definition}

In light of assumption $(g_{4})$, the functions $G_{x, y}$ and $\widehat{G}_x$ satisfy the $\Delta_2$-condition (see \cite[Proposition 2.3]{mmrv2008neu}), written $G_{x, y} \in \Delta_2$ and $\widehat{G}_x \in \Delta_2$, that is there exists a positive constant $K$ such that
\begin{equation}
G_{x, y}(2 t) \leqslant K G_{x, y}(t) \quad \textsl{ for all  }\quad (x, y) \in \Omega \times \Omega \quad \textsl{ and  }\quad t > 0 \text {, }
\end{equation}
and
\begin{equation}\label{1.6eq2}
\widehat{G}_x(2 t) \leqslant K \widehat{G}_x(t) \quad \textsl{ for all  }\quad x \in \Omega \quad \textsl{ and  }\quad t > 0.
\end{equation}

For technical reasons, let us assume, throughout this paper, that :
\begin{itemize}

\item[$(H_{1})$] $G_{x,y}$ and $\widehat{G}_{x}$ are locally integrables, that is for any constant number $c > 0$ and for every compact set $A \subset \Omega$ we have : 

\begin{equation}\label{eq1.3}
\int_{A \times A}G_{x,y}(c)dxdy < +\infty \textsl{ and } \int_{A}\widehat{G}_{x}(c)dx<+\infty,
\end{equation}
\item[$(H_{2})$] 
\begin{equation}\label{eq1.4}
G(x-z,y-z,t)=G(x,y,t) \quad\quad \forall (x,y),(z,z)\in \Omega \times \Omega \textit{, } \forall t \geqslant 0.
\end{equation}

\end{itemize}

\begin{definition}
We say that a generalized N-function $G_{x,y}$ satisfies the fractional boundedness condition, writen $G_{x,y}\in \mathcal{B}_{f} $, if there exist $C_{1}, C_{2} > 0$ such that
\begin{equation}
 C_{1}\leqslant G_{x,y}(1)\leqslant C_{2} \quad \forall (x,y)\in \Omega\times \Omega.
 \end{equation}
\end{definition}

We have the following examples of generalized N-functions satisfying the previous assumptions, and thus are admissible in our results on density of smooth functions.
\begin{itemize}
\item[(1)] Let $G_{x, y}(t)=t^{p(x, y)}$, for all $(x, y) \in \Omega \times \Omega$ and all $t \geqslant 0$, where $p: \Omega \times \Omega \longrightarrow(1,+\infty)$ is a continuous function satisfying
$$
1<p^{-} \leq p(x, y) \leq p^{+}<+\infty, \text {for all }(x, y) \in \Omega \times \Omega
$$
and
$$p\left( (x, y)-(z,z)\right)=p(x, y),\quad \text {for all }\quad(x, y),(z,z) \in \Omega \times \Omega.  $$
In this case the function $G_{x,y}$ satisfies the assumptions $(g_{1})-(g_{4})$, $(H_{1})$ and $(H_{2})$ and $G_{x,y}\in \mathcal{B}_{f}$.
\item[(2)] Let $G_{x, y}(t)=M(t)$, for all $(x, y) \in \Omega \times \Omega$ and all $t \geqslant 0$, where $\displaystyle M(t):=\int_0^t m(\tau) d \tau$ is an N-function (for definition see \cite{Adams1975}]) satisfying the following condition
$$1<m^{-} \leq \frac{m(t) t}{M(t)} \leq m^{+}<+\infty \quad \textsl{ for all }\quad t>0.$$
It is clear that the generalized N-function $G_{x, y}$ satisfies the assumptions $(g_{1})-(g_{4})$, $(H_{1})$ and $(H_{2})$ and $G_{x,y}\in \mathcal{B}_{f}$.
\end{itemize}
\begin{definition}
Let $G_{x,y}$ be a generalized N-function. The function $\widetilde{G}: \Omega \times \Omega \times [0,+\infty) \rightarrow[0,+\infty)$ defined by
\begin{equation}\label{Gconjugate}
\widetilde{G}_{x, y}(t)=\widetilde{G}(x, y, t):=\sup _{s \geq 0}\left(t s-G_{x, y}(s)\right) \quad  \forall  (x, y) \in \Omega \times \Omega, \quad  \forall t \geqslant 0
\end{equation}
is called the conjugate of $G$ in the sense of Young.
\end{definition}
It is not hard to see that $\left(g_1\right)-\left(g_4\right)$ imply that $\widetilde{G}$ is a generalized $N$-function and satisfies the $\Delta_2$-condition. Moreover, in view of \eqref{Gconjugate} we have the following Young's type inequality:
\begin{equation}\label{ineqYong}
\sigma \tau \leq G_{x, y}(\sigma)+\widetilde{G}_{x, y}(\tau), \text { for all }(x, y) \in \Omega \times \Omega \text { and } \sigma, \tau \geq 0.
\end{equation}
\subsection{Main results}

\begin{definition}

The open set $\Omega \subseteq \mathbb{R}^N$ is a hypograph if there exists a continuous function $\xi: \mathbb{R}^{N-1} \rightarrow \mathbb{R}$ such that, up to a rigid motion,
$$
\Omega:=\left\{\left(x^{\prime}, x_N\right) \in \mathbb{R}^{N-1} \times \mathbb{R}: x_N<\xi\left(x^{\prime}\right)\right\} .
$$
\end{definition}

 The first main result is the following Theorem.
\begin{theorem}\label{theo1}
 Let $\Omega \subseteq \mathbb{R}^N$ be a hypograph. Assume that $(g_{1})-(g_{4})$ and $(H_{1})-(H_{2})$ hold and $G_{x,y}\in \mathcal{B}_{f}$. Then, the space $C_0^{\infty}(\Omega)$ is dense in $W^{s,G_{x,y}}_{0}(\Omega)$.
\end{theorem}

\begin{definition}
The open set $\Omega \subseteq \mathbb{R}^N$ is a domain with continuous boundary $\partial \Omega$ if the following conditions are satisfied:
\begin{itemize}
\item[(a)]$\partial \Omega$ is compact
\item[(b)] there exist $M \in \mathbb{N}$, open sets $W_1, \ldots, W_M \subseteq \mathbb{R}^N$, sets $\Omega_1, \ldots, \Omega_M \subseteq \mathbb{R}^N$, continuous functions $\xi_1, \ldots, \xi_M: \mathbb{R}^{N-1} \rightarrow \mathbb{R}$ and rigid motions $\mathscr{R}_1, \ldots, \mathscr{R}_M$ : $\mathbb{R}^N \rightarrow \mathbb{R}^N$ satisfying the following conditions:
\begin{itemize}
\item[$(b_{1})$] $\partial \Omega \subseteq \bigcup_{j=1}^M W_j,$
\item[$(b_{2})$]$\mathscr{R}_j\left(\Omega_j\right):=\left\{\left(x^{\prime}, x_N\right) \in \mathbb{R}^{N-1} \times \mathbb{R}: x_N<\xi_j\left(x^{\prime}\right)\right\}, \text { for any } j \in\{1, \ldots, M\},$
\item[$(b_{3})$] $W_j \cap \Omega=W_j \cap \Omega_j .$
\end{itemize}
\end{itemize}
\end{definition}

The second main result is the following Theorem.
\begin{theorem}\label{theo2}
Let $\Omega$ be an open subset of $\mathbb{R}^N$ with continuous boundary. Assume that $(g_{1})-(g_{4})$ and $(H_{1})-(H_{2})$  hold and $G_{x,y}\in \mathcal{B}_{f}$. Then, the space $C_0^{\infty}(\Omega)$ is  dense in $W^{s,G_{x,y}}_{0}(\Omega)$.
\end{theorem}

This paper is organized as follows.  In Section \ref{sec2}, we state some fundamental properties of the generalized N-functions, Musielak-Orlicz spaces and fractional Musielak-Sobolev spaces. In Section \ref{sec3} we give the proof of the main results.
\section{ Some preliminaries results}\label{sec2}

In this section we give some definitions and properties for the fractional Musielak-Sobolev spaces and prove some preliminary lemmas, which will be used in the sequel.

\subsection{Musielak-Orlicz spaces}
Let $G_{x,y}$ be a generalized N-function. In correspondence to $\widehat{G}_x=G_{x,x}$ and an open subset $\Omega$ of $\mathbb{R}^{N}$, the Musielak class is defined as follows
$$
K^{\widehat{G}_x}(\Omega)=\left\{u: \Omega \longrightarrow \mathbb{R} \text { mesurable : } \int_{\Omega} \widehat{G}_x(|u(x)|) \mathrm{d} x<+\infty\right\} \text {, }
$$
and the Musielak-Orlicz space $L^{\widehat{G}_x}(\Omega)$ is defined as follows
$$
L^{\widehat{G}_x}(\Omega)=\left\{u: \Omega \longrightarrow \mathbb{R} \text { mesurable : } J_{\widehat{G}_x}(\lambda u)<+\infty \text { for some } \lambda>0\right\},
$$
where the modular $J_{\widehat{G}_x}$ is defined as  $$J_{\widehat{G}_x}( u):=\displaystyle\int_{\Omega} \widehat{G}_x(|u(x)|) \mathrm{d} x.$$\\
The space $L^{\widehat{G}_x}(\Omega)$ is endowed with the Luxemburg norm
$$
\|u\|_{L^{\widehat{G}_{x}}(\Omega)}=\inf \left\{\lambda>0: J_{\widehat{G}_x}\left(\frac{u}{\lambda} \right)  \leqslant 1\right\} .
$$
We would like to mention that the assumptions $(g_{1})-(g_{4})$ ensure that $\left(L^{\widehat{G}_{x}}(\Omega),\|u\|_{L^{\widehat{G}_{x}}(\Omega)}\right)$ is a separable
and reflexive Banach space.\\
The relation \eqref{1.6eq2} implies that $ L^{\widehat{G}_x}(\Omega)=K^{\widehat{G}_x}(\Omega)$ (see \cite{musju1983Orlicz}).

As a consequence of \eqref{ineqYong}, we have the following Hölder's type inequality:
\begin{lemma}
 Let $\Omega$ be an open subset of $\mathbb{R}^N$. Let $\widehat{G}_x$ be a generalized $N$-function and $\widetilde{\widehat{G}}_x$ its conjugate function, then
$$
\displaystyle\left|\int_{\Omega} u v d x\right| \leq 2\|u\|_{L^{\widetilde{G}_x}(\Omega)}\|v\|_{L^{\widetilde{\widehat{G}}}(\Omega)},
$$
for all $u \in L^{\widehat{G}_x}(\Omega)$ and all $ v \in L^{\widetilde{\widehat{G}}}(\Omega).$
\end{lemma}
\subsection{Fractional Musielak-Sobolev spaces}
Let $G_{x,y}$ be a generalized N-function, $s\in (0,1)$ and $\Omega$ an open subset of $\mathbb{R}^{N}$, we define the fractional Musielak-Sobolev space $W^{s,G_{x, y}}(\Omega)$ as follows
$$
\begin{aligned}
&W^{s,G_{x, y}}(\Omega):=\left\{u \in L^{\widehat{G}_x}(\Omega): J_{s,G_{x, y}}( \lambda u)<+\infty \quad\textsl{ for some } \lambda>0 \right\},
\end{aligned}
$$
where the modular $J_{s,G_{x, y}}$ is defined as  $$J_{s,G_{x, y}}(u):= \int_{\Omega} \int_{\Omega} G_{x, y}\left(D_{s}u(x,y)\right)d\mu,$$
with $D_{s}u(x,y):=\displaystyle\frac{|u(x)-u(y)|}{|x-y|^{s}}$ and $d\mu:=\displaystyle\frac{\mathrm{d} x \mathrm{~d} y}{|x-y|^{N}}.$\\
It is well known that $d\mu$ is a regular Borel measure on the set $\Omega\times\Omega$.

The space $W^{s,G_{x, y}}(\Omega)$ is endowed with the norm
\begin{equation}\label{Norm}
\|u\|_{W^{s,G_{x, y}}(\Omega)}:=\|u\|_{L^{\widehat{G}_{x}}(\Omega)}+[u]_{s, G_{x, y}},
\end{equation}
where $[.]_{s, G_{x, y}}$ is the so called $(s,G_{x, y})$-Gagliardo seminorm defined by
$$
[u]_{s, G_{x, y}}=\inf \left\{\lambda>0: J_{s,G_{x, y}}\left( \frac{u}{\lambda}\right)  \leqslant 1\right\} .
$$

\begin{remark}[{see \cite[Theorem~2.1]{abss2022class}}] Since the assumption $(g_{4})$ implies that the function $G_{x, y}$ and $\widetilde{G}_{x, y}$ satisfy the $\Delta_{2}$-condition, then the space $W^{s,G_{x, y}}(\Omega)$ is a reflexive and separable Banach space. Moreover, if $t\rightarrow G_{x, y}(\sqrt{t})$ is convex on $[0,+\infty)$, then the space $W^{s,G_{x, y}}(\Omega)$ is an uniformly convex space.
\end{remark}
Now, let us recall the following technical and important results.
\begin{lemma}[{\cite[Lemma 2.2]{abss2022class}}]\label{lemmapro}
Suppose that the assumptions $(g_{1})-(g_{4})$ hold. Then, the function $G_{x, y}$ satisfies the following inequalities  $:$
\begin{equation}
G_{x, y}(\delta t) \geqslant \delta^{g^{-}} G_{x, y}(t) \quad \quad \forall t>0,\quad \forall\delta>1,
\end{equation}
\begin{equation}
G_{x, y}(\delta t) \geqslant \delta^{g^{+}} G_{x, y}(t),\quad \forall t>0, \quad \forall\delta \in(0,1),
\end{equation}
\begin{equation}
G_{x, y}(\delta t) \leqslant \delta^{g^{+}} G_{x, y}(t) \quad \forall t>0,\quad \forall\delta>1,
\end{equation}
\begin{equation}
G_{x, y}(\delta t) \leqslant \delta^{g^{-}} G_{x, y}\left(t\right)\quad \forall t>0,\quad \forall \delta \in(0,1).
\end{equation}
\end{lemma}

\begin{proposition}[{\cite[Theorem~3.35]{vigelis2011musielak}}]\label{2prop1} Let $\widehat{G}_{x}$ be a generalized N-function satisfying the $\Delta_{2}$-condition, and $u,u_{n}\in L^{\widehat{G}_{x}}(\Omega)$, $n\in \mathbb{N}$. Then the following statements are equivalent
\begin{itemize}
\item[(i)] $\displaystyle\lim_{n\longrightarrow +\infty} \|u_{n}-u\|_{L^{\widehat{G}_{x}}(\Omega)}=0.$

\item[(ii)] $\displaystyle \lim_{n\longrightarrow +\infty} \int_{\Omega}\widehat{G}_{x}(|u_{n}-u|)dx=0.$
\end{itemize} 

\end{proposition}

\begin{proposition}\label{2prop11} Let $G_{x,y}$ be a generalized N-function satisfying the $\Delta_{2}$-condition, and $u,u_{n}\in W^{s,G_{x,y}}(\Omega)$, $n\in \mathbb{N}$. Then the following statements are equivalent
\begin{itemize}
\item[(i)] $\displaystyle\lim_{n\longrightarrow +\infty} [u_{n}-u]_{s, G_{x, y}}=0.$

\item[(ii)] $\displaystyle \lim_{n\longrightarrow +\infty} J_{s,G_{x, y}}(u_{n}-u)=0.$
\end{itemize} 

\end{proposition}

The proof of Proposition \ref{2prop11} follows similar techniques as those used in proof of Proposition \ref{2prop1}.

\begin{proposition}[{\cite[Remark~2.1]{ayya2020someapp}}]\label{2prop2} Suppose that $\widehat{G}_{x}$ is a generalized N-function satisfying \eqref{eq1.3} and the $\Delta_{2}$-condition. Then $C^{\infty}_{0}(\Omega)$ is dense in $ (L^{\widehat{G}_{x}}(\Omega), \|.\|_{L^{\widehat{G}_{x}}(\Omega)})$.

\end{proposition}

Now, let us show the following lemma.
\begin{lemma}\label{lemreg}
Let $\Omega$ be an open subset of $\mathbb{R}^{N}$ and let $u \in C_0^{\infty}\left(\Omega\right)$. Assume that $(g_{1})-(g_{4})$ hold and $G_{x,y}\in \mathcal{B}_{f}$. Then
$$
\int_{\Omega} \int_{\Omega} G_{x, y}\left( \frac{|u(x)-u(y)|}{|x-y|^{s}}\right)  \frac{dx dy}{|x-y|^{N}} < +\infty .
$$

\end{lemma}
\begin{proof}

 Let $u \in C_0^\infty(\Omega)$, then $u \in L^{\widehat{G}_x}(\Omega)$. On the one hand, we have
$$
|u(x)-u(y)| \leqslant\|\nabla u\|_{L^{\infty}(\Omega)}|x-y| \quad \text { and } \quad|u(x)-u(y)| \leqslant\left. 2\| u\right\|_{L^{\infty}(\Omega)} .
$$
Accordingly, we get
$$
|u(x)-u(y)| \leqslant 2\| u\|_{C^1(\Omega)} \min \{1,|x-y|\}:=\beta \delta(x, y),
$$
with $\beta=2|| u \|_{C^1(\Omega)}$ and $\delta(x, y)=\min \{1,|x-y|\}$. Then, by Lemma \ref{lemmapro} we have
$$
\begin{aligned}
& \int_{\Omega} \int_{\Omega} G_{x, y}\left(\frac{|u(x)-u(y)|}{|x-y|^s}\right) \frac{\mathrm{d} x \mathrm{~d} y}{|x-y|^N} \quad \leqslant \int_{\Omega} \int_{\Omega} G_{x, y}\left(\frac{\beta\delta(x, y)}{|x-y|^s}\right) \frac{\mathrm{d} x \mathrm{~d} y}{|x-y|^N} \\
& \quad \leqslant C\int_{\Omega} \int_{\Omega \cap|x-y| \leqslant 1} G_{x, y}\left(\frac{\delta(x, y)}{|x-y|^s}\right) \frac{\mathrm{d} x \mathrm{~d} y}{|x-y|^N}\\
&+C\int_{\Omega} \int_{\Omega \cap|x-y| \geqslant 1} G_{x, y}\left(\frac{\delta(x, y)}{|x-y|^s}\right) \frac{\mathrm{d} x \mathrm{~d} y}{|x-y|^N} \\
& \quad=C\int_{\Omega} \int_{\Omega \cap|x-y| \leqslant 1} G_{x, y}\left(\frac{|x-y|}{|x-y|^s}\right) \frac{\mathrm{d} x \mathrm{~d} y}{|x-y|^N}\\
&+C\int_{\Omega} \int_{\Omega \cap|x-y| \geqslant 1} G_{x, y}\left(\frac{1}{|x-y|^s}\right) \frac{\mathrm{d} x \mathrm{~d} y}{|x-y|^N}
\end{aligned}
$$
$$
\begin{aligned}
& \leqslant C\sup _{(x, y) \in \Omega \times \Omega} G_{x, y}(1)\left(\int_{\Omega} \int_{\Omega \cap|x-y| \leqslant 1} \frac{\mathrm{d} x \mathrm{~d} y}{|x-y|^{N+s-1}}+\int_{\Omega} \int_{\Omega \cap|x-y| \geqslant 1} \frac{\mathrm{d} x \mathrm{~d} y}{|x-y|^{N+s}}\right) \\
& =C\sup _{(x, y) \in \Omega \times \Omega} G_{x, y}(1)\left(I_1+I_2\right) .
\end{aligned}
$$
where the constant $C$ depends on $\beta$, $g^{+}$ and $g^{-}$.
On the other hand, since $N+s-1<N$, then the kernel $|x-y|^{-(N+s-1)}$ is summable with respect to $y$ if $|x-y| \leqslant 1$, hence $I_1$ is finite. On the other hand, as $N+s>N$, it follows that the kernel \\$\mid x-$ $\left.y\right|^{-(N+s)}$ is summable when $|x-y| \geqslant 1$, thus, $I_2$ is finite. Consequently, the two integrals above are finite, which complete the proof.
\end{proof}

Now we give two approximation results.

\begin{lemma}\label{2lem2}
 Let $u \in L^{\widehat{G}_{x}}\left(\mathbb{R}^N\right)$. Then there exists a sequence of functions $u_n \in$ $L^{\widehat{G}_{x}}\left(\mathbb{R}^N\right) \cap L^{\infty}\left(\mathbb{R}^N\right)$ such that
$$
\left\|u-u_n\right\|_{L^{\widehat{G}_{x}}\left(\mathbb{R}^N\right)} \longrightarrow 0 \quad \text { as } \quad n \longrightarrow+\infty .
$$
\end{lemma}
\begin{proof}
 Let us set
$$
u_n(x):= \begin{cases}n & \text { if } u(x) \geq n, \\ u(x) & \text { if } u(x) \in(-n, n), \\ -n & \text { if } u(x) \leq-n .\end{cases}
$$
We have
$$
u_n \longrightarrow u \text { a.e. in } \mathbb{R}^N
$$
and
$$
\widehat{G}_{x}\left( \left|u_n(x)\right|\right)  \leq \widehat{G}_{x}\left( |u(x)|\right) \in L^1\left(\mathbb{R}^N\right) .
$$
Hence the result follows from the dominated convergence theorem and Proposition \ref{2prop1}.
\end{proof}
\begin{lemma}\label{2lem3}
 Let $u \in L^{G_{x,y}}\left(\mathbb{R}^N \times \mathbb{R}^N,d\mu\right)$. Then there exists a sequence of functions $u_k \in L^{G_{x,y}}\left(\mathbb{R}^N \times \mathbb{R}^N,d\mu\right) \cap L^{\infty}\left(\mathbb{R}^N \times \mathbb{R}^N,d\mu\right)$ such that
$$
\left\|u-u_k\right\|_{L^{G_{x,y}}\left(\mathbb{R}^N \times \mathbb{R}^N,d\mu\right)} \longrightarrow 0 \quad \text { as } \quad k \longrightarrow+\infty \text {. }
$$
\end{lemma}
\begin{proof}
 Let us set
$$
u_k(x, y):= \begin{cases}k & \text { if } u(x, y) \geq k \\ u(x, y) & \text { if } u(x, y) \in(-k, k), \\ -k & \text { if } u(x, y) \leq-k.\end{cases}
$$
We have
$$
u_k \longrightarrow u \text { a.e. in } \mathbb{R}^N \times \mathbb{R}^N
$$
and
$$
G_{x,y}\left(\left|u_k(x, y)\right|\right) \leq G_{x,y}\left(|u(x, y)|\right) \in L^1\left(\mathbb{R}^N \times \mathbb{R}^N,d\mu\right) .
$$
Thus by the dominated convergence theorem and Proposition \ref{2prop1}, the claim follows.
\end{proof}

 From now on, $\mathcal{B}_c(\Omega)$ will stand for the set of bounded functions compactly supported in $\Omega$, $B_{R}$ will denote the ball centered at $0$ with radius $R>0$.

For $h \in \mathbb{R}^N$, let $T_h u$ stand for the translation operator defined by
$$
T_h u(x)= \begin{cases}u(x+h) & \text { if } x \in \Omega \text { and } x+h \in \Omega, \\ 0 & \text { otherwise in } \mathbb{R}^N .\end{cases}
$$

If the function $u$ has a compact support, $T_h u$ is well-defined provided that $|h|<$ $\operatorname{dist}(\operatorname{Supp} u, \partial \Omega$ ).

\begin{theorem} [{\cite[Theorem~2.1]{ayya2020someapp}}]\label{Theo33}
 Let $\widehat{G}_{x}$ be a generalized N-function satisfying \eqref{eq1.3}. Let $u \in \mathcal{B}_c(\Omega)$, then for every $\varepsilon>0$ there exists an $\eta=\eta(\varepsilon)>0$ such that for $h \in \mathbb{R}^N$ with $|h|<\eta$ we have
$$
\left\|T_h u-u\right\|_{L^{\widehat{G}_{x}}(\Omega)}<\varepsilon .
$$
\end{theorem}

\begin{remark}[ {\cite[Remark~2.2]{ayya2020someapp}}]
The boundedness of the function $u$ in Theorem \ref{Theo33} is
necessary, else the result is false. Indeed, for the particular case $\widehat{G}_{x}(t)=t^{p(x)}$, the authors in \cite{kovavcik1991spaces} consider the following example: $N=1, \Omega=(-1,1)$. For $1 \leqslant r<d<+\infty$ they define the variable exponent
$$
p(x)= \begin{cases}r & \text { if } x \in[0,1), \\ d & \text { if } x \in(-1,0)\end{cases}
$$
and consider the function
$$
f(x)= \begin{cases}x^{-1 / d} & \text { if } x \in[0,1), \\ 0 & \text { if } x \in(-1,0) .\end{cases}
$$

They show that $T_h f \notin L^{p(\cdot)}(\Omega)$ although $f \in L^{p(\cdot)}(\Omega)$. note that, in this example, the function $f$ is compactly supported but not bounded on $\Omega$.

\end{remark}
The proofs of Theorem \ref{theo1} and Theorem \ref{theo2} are mainly based on a basic technique of convolution (which makes functions $C^{\infty}$), joined with a cut-off (which makes their support compact). Here we will give some properties of these operations with respect to the norm in \eqref{Norm}.\\

Let $J$ stand for the Friedrichs mollifier kernel defined on $\mathbb{R}^N$ by
$$
J(x)= \begin{cases}k \mathrm{e}^{-1 /\left(1-|x|^2\right)} & \text { if }|x|<1, \\ 0 & \text { if } |x| \geqslant 1,\end{cases}
$$
where $k>0$ is such that $\displaystyle\int_{\mathbb{R}^{N}} J(x) \mathrm{d} x=\int_{B_{1}} J(x) \mathrm{d} x=1$.\\
 For $\varepsilon>0$, we define $J_{\varepsilon}(x)=\varepsilon^{-N} J\left(x/ \varepsilon\right)$ and for any $u \in W^{s,G_{x,y}}\left(\mathbb{R}^N\right)$, let us denote by $u_{\varepsilon}$ the function defined as the convolution between $u$ and $J_{\varepsilon}$ ; that is,
$$
u_{\varepsilon}(x)=(u*J_{\varepsilon})(x)=\int_{\mathbb{R}^N} J_{\varepsilon}(x-y) u(y) \mathrm{d} y=\int_{B_{1}} u(x-\varepsilon y) J(y) \mathrm{d} y.
$$
Of course, by construction, $u_{\varepsilon}$ is a smooth function, i.e. $u_{\varepsilon} \in C^{\infty}\left(\mathbb{R}^N\right)$. On the other hand, if $u$ is supported in $\Omega$ it is not possible, in general, to conclude that $u_{\varepsilon} \in W_0^{s, G_{x,y}}(\Omega)$, since the support of $u_{\varepsilon}$ may exceed the one of $u$ and so it may exit $\Omega$. 
\begin{lemma}[{\cite[ Corollary 2.1]{ayya2020someapp}}]\label{lemmacorollary}  Let $\widehat{G}_{x}$ be a generalized N-function satisfying \eqref{eq1.3} and let $u \in \mathcal{B}_c(\Omega)$. For any $\varepsilon>0$ small enough, we have $u_{\varepsilon} \in C_0^{\infty}(\Omega)$. Moreover,
$$
\left\|u_{\varepsilon}-u\right\|_{L^{\widehat{G}_{x}}(\Omega)} \rightarrow 0 \quad \text { as } \varepsilon \rightarrow 0 .
$$
\end{lemma}
\begin{lemma}[{\cite[ Lemma B.1]{ayya2020someapp}}]\label{lemma22}
Let $\widehat{G}_{x}$ be a generalized N-function satisfying \eqref{eq1.3} and the $\Delta_{2}$-condition. Then
$\mathcal{B}_c(\Omega)$ is dense in $(L^{\widehat{G}_{x}}(\Omega), \|u\|_{L^{\widehat{G}_{x}}(\Omega)})$.
\end{lemma}
\begin{lemma}\label{lem2.06}  Suppose that $(g_{1})-(g_{4})$ and $(H_{1})-(H_{2})$ hold. Let $u \in W^{s,G_{x,y}}\left(\mathbb{R}^N\right)$. Then $\left\|u-u_{\varepsilon}\right\|_{W^{s,G_{x,y}}\left(\mathbb{R}^N\right)} \rightarrow 0$ as $\varepsilon \rightarrow 0$.
\end{lemma}
\begin{proof}
 Let $u \in W^{s,G_{x,y}}\left(\mathbb{R}^N\right)$, then $u \in L^{\widehat{G}_{x}}(\mathbb{R}^{N})$.
Thus by Lemma \ref{lemma22}, we can assume that $u$ is bounded and compactly supported in $\mathbb{R}^{N}$. Then by Lemma \ref{lemmacorollary}, we have
\begin{equation}\label{eq2.01111}
\left\|u-u_{\varepsilon}\right\|_{L^{\widehat{G}_{x}}(\mathbb{R}^{N})} \rightarrow 0 \quad \text { as } \quad \varepsilon \rightarrow 0 .
\end{equation}
Thus, by Proposition \ref{2prop11}, it suffices to prove that

\begin{equation}\label{eq2.14}
J_{s,G_{x, y}}(u-u_\varepsilon) \longrightarrow 0, \quad \textsf{ as } \quad \varepsilon \longrightarrow 0.
\end{equation}
 By the definition of $u_{\varepsilon}$, the Jensen’s inequality, Tonelli's and Fubini's theorems, we have
\begin{equation}\label{eq2.0017}
\begin{aligned}
&J_{s,G_{x, y}}(u-u_\varepsilon)\\
& =\int_{\mathbb{R}^{N} \times \mathbb{R}^{N}}G_{x,y}\left(\frac{|u_\varepsilon(x)-u(x)-u_\varepsilon(y)+u(y)|}{|x-y|^{s}}\right)  \frac{d x d y}{|x-y|^{N}}  \\
& =\int_{\mathbb{R}^{N} \times \mathbb{R}^{N}} G_{x,y}\left(\frac{\vert\int_{\mathbb{R}^N}(u(x-z)-u(y-z)) J_\varepsilon(z) d z-u(x)+u(y)\vert}{|x-y|^{s}}\right)\frac{d x d y}{|x-y|^{N}} \\
& =\int_{\mathbb{R}^{N} \times \mathbb{R}^{N}} G_{x,y}\left(\frac{|\int_{B_1}\left(u(x-\varepsilon z)-u(y-\varepsilon z)-u(x)+u(y)\right) J(z) d z|}{|x-y|^{s}}\right)\frac{d x d y}{|x-y|^{N}}\\
&\leq \int_{\mathbb{R}^{N} \times \mathbb{R}^{N}}\left[\int_{B_1}G_{x,y}\left( \frac{|u(x-\varepsilon z)-u(y-\varepsilon z)-u(x)+u(y)| J(z)}{|x-y|^{s}}\right) d z\right] \frac{d x d y}{|x-y|^{N}} \\
& \leq \int_{\mathbb{R}^{N} \times \mathbb{R}^{N} \times B_1}G_{x,y}\left( \frac{|u(x-\varepsilon z)-u(y-\varepsilon z)-u(x)+u(y)|}{|x-y|^{s}}\right)\\
&\times \left(J(z)^{g^{+}}+J(z)^{g^{-}}\right) \frac{d x d y}{|x-y|^{N}} .
\end{aligned}
\end{equation}
We claim that
\begin{equation} \label{eqclaim}
\int_{\mathbb{R}^{N} \times \mathbb{R}^{N}}G_{x,y}\left(\frac{|u(x-\varepsilon z)-u(y-\varepsilon z)-u(x)+u(y)|}{|x-y|^{s}}\right) \frac{d x d y}{|x-y|^{N}}\longrightarrow 0,
\end{equation}
as $\varepsilon \longrightarrow 0$.\\
Fix $z \in B_1$ and put $w:=(z, z) \in \mathbb{R}^{N} \times \mathbb{R}^{N}$. We define the function\\ $v: \mathbb{R}^{N} \times \mathbb{R}^{N} \rightarrow \mathbb{R}$ by
$$
v(x, y):=\frac{(u(x)-u(y))}{|x-y|^{s}}, \quad \forall(x, y) \in \mathbb{R}^{N} \times \mathbb{R}^{N} .
$$
Sine $u \in W^{s,G_{x,y}}\left(\mathbb{R}^N\right)$, then $v \in L^{G_{x,y}}(\mathbb{R}^{N} \times \mathbb{R}^{N},d\mu)$. If $\varepsilon^{\prime}>0$, by Proposition \ref{2prop2}, there exists $\mathrm{g} \in C_0^{\infty}(\mathbb{R}^{N} \times \mathbb{R}^{N})$ with $\displaystyle\|v-g\|_{L^{G_{x,y}}(\mathbb{R}^{N} \times \mathbb{R}^{N},d\mu)}<\frac{\varepsilon^{\prime}}{3}$, then
$$
\begin{aligned}
& \|v(.-\varepsilon w)-v\|_{L^{G_{x,y}}(\mathbb{R}^{N} \times \mathbb{R}^{N},d\mu)} \\
\leq & \|v(.-\varepsilon w)-g(.-\varepsilon w)\|_{L^{G_{x,y}}(\mathbb{R}^{N} \times \mathbb{R}^{N},d\mu)}+\|g(.-\varepsilon w)-g\|_{L^{G_{x,y}}(\mathbb{R}^{N} \times \mathbb{R}^{N},d\mu)}\\
&+\|v-g\|_{L^{G_{x,y}}(\mathbb{R}^{N} \times \mathbb{R}^{N},d\mu)} \\
\leq & \frac{\varepsilon^{\prime}}{3}+\frac{\varepsilon^{\prime}}{3}+\frac{\varepsilon^{\prime}}{3}=\varepsilon^{\prime},
\end{aligned}
$$
with $\varepsilon$ is sufficiently small. This proves our claim \eqref{eqclaim}.\\
Moreover, for a.e. $z \in B_1$, by Lemma \ref{lemmapro} and the convexity of $G_{x,y}$ we have
\begin{equation}\label{eq2.3}
\begin{aligned}
& \left(J(z)^{g^{+}}+J(z)^{g^{-}}\right) \int_{\mathbb{R}^{N} \times \mathbb{R}^{N}}G_{x,y}\left(\frac{|u(x-\varepsilon z)-u(y-\varepsilon z)-u(x)+u(y)|}{|x-y|^{s}}\right) \frac{d x d y}{|x-y|^{N}} \\
& \leq 2^{g^{+}-1}\left(J(z)^{g^{+}}+J(z)^{g^{-}}\right)\left(\int_{\mathbb{R}^{N} \times \mathbb{R}^{N}}G_{x,y}\left(\frac{|u(x-\varepsilon z)-u(y-\varepsilon z)|}{|x-y|^{s}}\right)\frac{d x d y}{|x-y|^{N}}\right. \\
& \left.+\int_{\mathbb{R}^{N} \times \mathbb{R}^{N}}G_{x,y}\left(\frac{|u(x)-u(y)|}{|x-y|^{s}}\right) \frac{d x d y}{|x-y|^{N}}\right) \\
& \leq 2^{g^{+}}\left(J(z)^{g^{+}}+J(z)^{g^{-}}\right) \int_{\mathbb{R}^{N} \times \mathbb{R}^{N}}G_{x,y}\left(\frac{|u(x)-u(y)|}{|x-y|^{s}}\right) \frac{d x d y}{|x-y|^{N}} \in L^{1}\left(B_1\right),
\end{aligned}
\end{equation}
for any $\varepsilon>0$. Then by this, \eqref{eqclaim} and the dominated convergence theorem, we have
\begin{equation}\label{limm2}
\begin{aligned}
&\int_{B_1} \int_{\mathbb{R}^{N} \times \mathbb{R}^{N}}G_{x,y}\left(\frac{|u(x-\varepsilon z)-u(y-\varepsilon z)-u(x)+u(y)|}{|x-y|^{s}}\right)\\&\left(J(z)^{g^{+}}+J(z)^{g^{-}}\right) \frac{d x d y}{|x-y|^{N}} d z \longrightarrow 0, \textsl{ as } \varepsilon\longrightarrow 0.
\end{aligned}
\end{equation}
Then by \eqref{limm2} and \eqref{eq2.0017}, we get \eqref{eq2.14}. This concludes the proof.
\end{proof}

\begin{lemma}\label{lem2.066} Assume that $(g_{1})-(g_{4})$ and $(H_{1})$ hold and let $u \in W^{s,G_{x,y}}\left(\mathbb{R}^N\right)$. Then $\left\|T_h u-u\right\|_{W^{s,G_{x,y}}\left(\mathbb{R}^N\right)} \rightarrow 0$ as $|h| \rightarrow 0$.
\end{lemma}
\begin{proof}
 Let $u \in W^{s,G_{x,y}}\left(\mathbb{R}^N\right)$, then $u \in L^{\widehat{G}_{x}}(\mathbb{R}^{N})$.
Thus according to Lemma \ref{lemmacorollary} and Lemma \ref{lemma22}, we can assume that $u\in C^{\infty}_{0}(\mathbb{R}^{N})$. So that by Theorem \ref{Theo33}, we have
\begin{equation}\label{eq2.01}
\left\|T_h u-u\right\|_{L^{\widehat{G}_{x}}(\mathbb{R}^{N})} \rightarrow 0 \quad \text { as } \quad |h|\rightarrow 0 .
\end{equation}
Thus, by Proposition \ref{2prop11}, it suffices to prove that 
\begin{equation}\label{eq2.1}
\displaystyle J_{s,G_{x, y}}(T_{h} u-u) \longrightarrow 0 \quad\quad \textsl{ as } \quad |h|\longrightarrow 0 .
\end{equation}
 We have
\begin{equation}\label{eq2.007}
\begin{aligned}
& J_{s,G_{x, y}}(T_{h} u-u)\\&=\int_{\mathbb{R}^{N} \times \mathbb{R}^{N}}G_{x,y}\left(\frac{|T_{h} u(x)-u(x)- T_{h}u(y)+u(y)|}{|x-y|^{s}}\right) \frac{dx dy}{|x-y|^{N}}  \\
& =\int_{\mathbb{R}^{N} \times \mathbb{R}^{N}}G_{x,y}\left(\frac{|(u(x+h)- u(y+h))-(u(x)-u(y))|}{|x-y|^{s}}\right) \frac{dx dy}{|x-y|^{N}}.
\end{aligned}
\end{equation}
Since $u\in C^{\infty}_{0}(\mathbb{R}^{N})$, then by the same way in the proof of Lemma  \ref{lemreg} we have 
\begin{equation}\label{ineqnew}
\begin{aligned}
&G_{x,y}\left(\frac{|(u(x+h)- u(y+h))-(u(x)-u(y))|}{|x-y|^{s}}\right) \\
&\leq G_{x,y}\left(\frac{2\beta \delta(x,y)}{|x-y|^{s}}\right)\in L^{1}(\mathbb{R}^{N} \times \mathbb{R}^{N},d\mu),
\end{aligned}
\end{equation}
where $\beta=2|| u \|_{C^1(\Omega)}$ and $\delta(x, y)=\min \{1,|x-y|\}$. \\
 Then combining \eqref{ineqnew}, \eqref{eq2.007} and the fact that \\
 $$G_{x,y}\left(\frac{|(u(x+h)- u(y+h))-(u(x)-u(y))|}{|x-y|^{s}}\right)\longrightarrow 0\quad \textsl{ as }\quad |h|\longrightarrow 0$$
which is given by the continuity of the functions $u$ and $t\rightarrow G_{x,y}(t)$, together with the dominated convergence Theorem, we have 
 $$\displaystyle J_{s,G_{x, y}}(T_{h} u-u) \longrightarrow 0 \quad\quad \textsl{ as } \quad |h|\longrightarrow 0 .$$
Then the result follows.
\end{proof}

Now, we will discuss the cut-off technique needed for the density argument. For any $j \in \mathbb{N}$, let $\tau_j \in C^{\infty}\left(\mathbb{R}^N\right)$ be such that
\begin{equation}\label{eqcut}
\begin{gathered}
0 \leq \tau_j(x) \leq 1, \quad \forall x \in \mathbb{R}^N, \\
\tau_j(x)= \begin{cases}1 & \text { if } x \in B_j, \\
0 & \text { if } x \in \mathbb{R}^N \backslash B_{j+1},\end{cases}
\end{gathered}
\end{equation}
and
$$ |\nabla\tau_j(x)|\leqslant C \quad \forall x\in \mathbb{R}^{N},$$
where $C$ is a positive constant not depending on $j$, and $B_j$ denotes the ball centered at $0$ with radius $j$.\\

We have the following result.
\begin{lemma}\label{2lem4} Assume that $(g_{1})-(g_{4})$ hold and  $G_{x,y}\in \mathcal{B}_{f}$. Let $u \in W^{s,G_{x,y}}\left(\mathbb{R}^N\right)$. Then $\tau_j u \in W^{s,G_{x,y}}\left(\mathbb{R}^N\right)$.
\end{lemma}
\begin{proof}
Let $u \in W^{s,G_{x,y}}\left(\mathbb{R}^N\right)$. Since $\left|\tau_j\right| \leq 1$ then $\tau_j u \in L^{\widehat{G}_{x}}\left(\mathbb{R}^N\right)$.\\ Furthermore, for some $\lambda>0$, we have
$$
\begin{aligned}
&\displaystyle J_{s,G_{x, y}}\left(\lambda\tau_ju\right)\\ 
& =\int_{\mathbb{R}^N} \int_{\mathbb{R}^N} G_{x,y}\left(\frac{\lambda\left|\tau_j(x) u(x)-\tau_j(y) u(y)\right|}{|x-y|^{s}}\right) \frac{d x d y}{|x-y|^{N}} \\
&\leq  2^{g^{+}-1} \int_{\mathbb{R}^N} \int_{\mathbb{R}^N} G_{x,y}\left(\frac{\lambda\left|\tau_j(x)(u(x)-u(y))\right|}{|x-y|^{s}}\right)  \frac{d x d y}{|x-y|^{N}} \\
& +2^{g^{+}-1} \int_{\mathbb{R}^N} \int_{\mathbb{R}^N} G_{x,y}\left(\frac{\lambda\left|u(y)\left(\tau_j(x)-\tau_j(y)\right)\right|}{|x-y|^{s}}\right)\frac{d x d y}{|x-y|^{N}}\\
&\leq  2^{g^{+}-1} \int_{\mathbb{R}^N} \int_{\mathbb{R}^N} G_{x,y}\left(\frac{\lambda|u(x)-u(y)|}{|x-y|^{s}}\right)\frac{d x d y}{|x-y|^{N}}\\
& +2^{g^{+}-1} \int_{\mathbb{R}^N} \int_{\mathbb{R}^N} G_{x,y}\left(\frac{\lambda\left|u(y)\left(\tau_j(x)-\tau_j(y)\right)\right|}{|x-y|^{s}}\right)\frac{d x d y}{|x-y|^{N}}.
\end{aligned}
$$
where
$$
\int_{\mathbb{R}^N} \int_{\mathbb{R}^N} G_{x,y}\left(\frac{\lambda|u(x)-u(y)|}{|x-y|^{s}}\right)\frac{d x d y}{|x-y|^{N}}<\infty
$$
 since $u \in W^{s,G_{x,y}}\left(\mathbb{R}^N\right)$.\\
By Lemma \ref{2lem2}, we can assume that $u \in L^{\infty}\left(\mathbb{R}^N\right)$. Therefore, by Lemma \ref{lemmapro} we have
\begin{equation}\label{ieqnew}
\begin{aligned}
&\int_{\mathbb{R}^N} \int_{\mathbb{R}^N} G_{x,y}\left(\frac{\lambda\left|u(y)\left(\tau_j(x)-\tau_j(y)\right)\right|}{|x-y|^{s}}\right)\frac{dx dy}{|x-y|^{N}} \\
&\leq C \int_{\mathbb{R}^N} \int_{\mathbb{R}^N} G_{x,y}\left(\frac{\left|\tau_j(x)-\tau_j(y)\right|}{|x-y|^{s}}\right)\frac{d x d y}{|x-y|^{N}},
\end{aligned}
\end{equation}
where the constant $C$ depends on $g^{+}$, $g^{-}$, $\lambda$ and $\|u\|_{L^{\infty}\left(\mathbb{R}^N\right)}$. Finally, since $ \tau_{j}\in C^{\infty}_{0}(\mathbb{R}^{N})$, then by Lemma \ref{lemreg} and the inequality \eqref{ieqnew}, we get
$$
\int_{\mathbb{R}^N} \int_{\mathbb{R}^N} G_{x,y}\left(\frac{\lambda\left|u(y)\left(\tau_j(x)-\tau_j(y)\right)\right|}{|x-y|^{s}}\right) \frac{d x d y}{|x-y|^{N}} <+\infty .
$$
This concludes the proof.
\end{proof}

\begin{lemma}\label{lem2.8}
Assume that $(g_{1})-(g_{4})$ hold and $G_{x,y}\in \mathcal{B}_{f}$. Let $u \in W^{s,G_{x,y}}\left(\mathbb{R}^N\right)$. Then  $\operatorname{Supp}\left(\tau_j u\right) \subseteq \overline{B}_{j+1} \cap \operatorname{Supp} u$, and
$$
\left\|\tau_j u-u\right\|_{W^{s,G_{x,y}}\left(\mathbb{R}^N\right)} \longrightarrow 0 \quad \text { as } j \longrightarrow+\infty .
$$
\end{lemma}
\begin{proof}
By \eqref{eqcut} and \cite[Lemma~9 ]{fiscella2015density}, we get
$$
\operatorname{Supp}\left(\tau_j u\right) \subseteq \overline{B}_{j+1} \cap \operatorname{Supp} u .
$$
Now, let us prove that
$$
\left\|\tau_j u-u\right\|_{W^{s,G_{x,y}}\left(\mathbb{R}^N\right)} \longrightarrow 0 \text { as } j \longrightarrow+\infty .
$$
From Proposition \ref{2prop1} and Proposition \ref{2prop11}, it suffices to prove that
$$
\int_{\mathbb{R}^N}\widehat{G}_{x}\left(\left|\tau_j(x) u(x)-u(x)\right|\right) d x \longrightarrow 0 \quad \text { as } j \longrightarrow+\infty
$$
and
$$
\displaystyle J_{s,G_{x, y}}(\tau_{j} u-u) \longrightarrow 0 \quad\quad \textsl{ as } \quad j\longrightarrow +\infty.
$$
Since $u\in L^{\widehat{G}_{x}}(\mathbb{R}^{N})$, we have
$$
\begin{aligned}
\widehat{G}_{x}\left(\left|\tau_j(x) u(x)-u(x)\right|\right) & \leq\widehat{G}_{x}\left(2|u(x)|\right) \\
& \leq 2^{g^{+}}\widehat{G}_{x}\left(|u(x)|\right)\in L^1\left(\mathbb{R}^N\right) .
\end{aligned}
$$
Moreover, by \eqref{eqcut} we have
$$
\widehat{G}_{x}\left(\left|\tau_j(x) u(x)-u(x)\right|\right) \longrightarrow 0 \quad \text { as } j \longrightarrow+\infty \text { a.e. in } \mathbb{R}^N.
$$
Then, by using the dominated convergence theorem, we get
$$
\int_{\mathbb{R}^N}\widehat{G}_{x}\left(\left|\tau_j(x) u(x)-u(x)\right|\right) \longrightarrow 0 \quad \text { as } j \longrightarrow+\infty .
$$
Now, let us show that
$$
\displaystyle J_{s,G_{x, y}}(\tau_{j} u-u) \longrightarrow 0 \quad\quad \textsl{ as } \quad j\longrightarrow +\infty.
$$
 We set $\eta_j=1-\tau_j$. Then $\eta_j u=u-\tau_j u$. Furthermore, we have
$$
\begin{aligned}
& \left|\tau_j(x) u(x)-u(x)-\tau_j(y) u(y)+u(y)\right| \\
& \quad=\left|\eta_j(x)(u(x)-u(y))-\left(\tau_j(y)-\tau_j(x)\right) u(y)\right| .
\end{aligned}
$$
Therefore, we have
\begin{equation}\label{inequation1}
\begin{aligned}
&J_{s,G_{x, y}}(\tau_{j} u-u)\\
&=\int_{\mathbb{R}^N} \int_{\mathbb{R}^N} G_{x,y}\left(\frac{\left|\tau_j(x) u(x)-u(x)-\tau_j(y) u(y)+u(y)\right|}{|x-y|^{s}}\right)\frac{d x d y}{|x-y|^{N}} \\
& \leq 2^{g^{+}-1} \int_{\mathbb{R}^N} \int_{\mathbb{R}^N} G_{x,y}\left(\frac{\left|\tau_j(x)-\tau_j(y)\right||u(y)|}{|x-y|^{s}}\right)\frac{d x d y}{|x-y|^{N}} \\
& \quad+2^{g^{+}-1} \int_{\mathbb{R}^N} \int_{\mathbb{R}^N} G_{x,y}\left(\frac{|u(x)-u(y)| \eta_j(x)}{|x-y|^{s}}\right) \frac{d x d y}{|x-y|^{N}} .
\end{aligned}
\end{equation}
According to Lemma \ref{2lem2}, we can suppose that $u \in L^{\infty}\left(\mathbb{R}^N\right)$. Therefore,
$$
\begin{aligned}
&G_{x,y}\left(\frac{\left|\tau_j(x)-\tau_j(y)\right||u(y)|}{|x-y|^{s}}\right)\\ &\leq C\left(\|u\|_{L^{\infty}\left(\mathbb{R}^N\right)}, g^{+}, g^{-}\right) G_{x,y}\left(\frac{\left|\tau_j(x)-\tau_j(y)\right|}{|x-y|^{s}}\right).
\end{aligned}
$$
By Lemma \ref{lemreg}, we have
$$
G_{x,y}\left(\frac{\left|\tau_j(x)-\tau_j(y)\right|}{|x-y|^{s}}\right) \in L^1\left(\mathbb{R}^N \times \mathbb{R}^N,d\mu\right) .
$$
Moreover, by \eqref{eqcut} we have
$$
G_{x,y}\left(\frac{\left|\tau_j(x)-\tau_j(y)\right||u(y)|}{|x-y|^{s}}\right)  \longrightarrow 0 \quad \text { as } j \longrightarrow \infty \text { a.e. in } \mathbb{R}^N \times \mathbb{R}^N.
$$
Hence, by using the dominated convergence theorem, we get
\begin{equation}\label{lim12}
\int_{\mathbb{R}^N} \int_{\mathbb{R}^N} G_{x,y}\left(\frac{\left|\tau_j(x)-\tau_j(y)\right||u(y)|}{|x-y|^{s}}\right)\frac{d x d y}{|x-y|^{N}} \longrightarrow 0 \quad \text { as } j \longrightarrow \infty .
\end{equation}
Also, by Lemma \ref{lemmapro}, we have
$$
G_{x,y}\left(\frac{|u(x)-u(y)|}{|x-y|^{s}}\eta_j(x)\right)\ \leq \eta_j(x)^{g^{-}} G_{x,y}\left(\frac{|u(x)-u(y)|}{|x-y|^{s}}\right),
$$
and since $u\in W^{s,G_{x,y}}\left(\mathbb{R}^N\right)$, then
$$
G_{x,y}\left(\frac{|u(x)-u(y)|}{|x-y|^{s}}\right) \in L^1\left(\mathbb{R}^N \times \mathbb{R}^N,d\mu\right).
$$
Again by \eqref{eqcut}, we have
$$
G_{x,y}\left(\frac{|u(x)-u(y)|}{|x-y|^{s}}\eta_j(x)\right) \longrightarrow 0 \quad \text { as } j \longrightarrow \infty \text { a.e. in } \mathbb{R}^N \times \mathbb{R}^N .
$$
Hence, by the dominated convergence theorem, we have
$$
\int_{\mathbb{R}^N} \int_{\mathbb{R}^N} G_{x,y}\left(\frac{|u(x)-u(y)|}{|x-y|^{s}}\eta_j(x)\right)\frac{d x d y}{|x-y|^{N}} \longrightarrow 0 \quad \text { as } j \longrightarrow \infty .
$$
Then by this, \eqref{lim12} and \eqref{inequation1} we get $$
\displaystyle J_{s,G_{x, y}}(\tau_{j} u-u) \longrightarrow 0 \quad\quad \textsl{ as } \quad j\longrightarrow +\infty.
$$
The proof is complete.
\end{proof}

For any $\delta > 0$ and any function $u$ we define the function $u_{\delta}$ by
$$u_{\delta}:=T_{(\bar{0},\delta)}u,$$
where $\bar{0}=(0,0,0,...,0)\in \mathbb{R}^{N-1}$.
\begin{lemma}[{\cite[Lemma 14]{fiscella2015density}}]\label{lemmad}
 Let $\Omega$ be a hypograph. Let $u: \mathbb{R}^N \rightarrow \mathbb{R}$ be such that $ u=0$ in $\mathbb{R}^{N}\setminus \Omega$.
Then,
\begin{equation}\label{eeeq1}
\text { Supp } u_\delta \subseteq \Omega \text {. }
\end{equation}
More precisely, given any $R>0$ there exists $a>0$ such that
\begin{equation}\label{eeeq2}
B_R \cap\left(\operatorname{Supp} u_\delta+B_a\right) \subseteq \Omega .
\end{equation}
The above quantity $a$ only depends on $N, u, \delta, R$ and $\Omega$ (say, $a=a(N, u, \delta, R, \Omega)$).

\end{lemma}

\section{Proofs of main results}\label{sec3}
 This section is aimed at proving  Theorem \ref{theo1} and Theorem \ref{theo2}.

\begin{proof}[Proof of Theorem~{\upshape\ref{theo1}}]
Let $\Omega$ be a hypograph and let $u \in W_{0}^{s,G_{x,y}}(\Omega)$. By possibly changing $u$ in a set
of zero measure, we suppose that
\begin{equation}\label{eeq1}
u=0 \text { in } \mathbb{R}^{N}\setminus \Omega \text {. }
\end{equation}
Let us fix $\sigma>0$. By Lemma \ref{lem2.066} there exists $\bar{\delta}=\bar{\delta}(\sigma)>0$ such that
\begin{equation}\label{eeq12}
\left\|u_{\delta}-u\right\|_{W^{s,G_{x,y}}(\mathbb{R}^{N})}<\frac{\sigma}{3}
\end{equation}
for $\delta$ sufficiently small, say $\delta \leqslant \bar{\delta}$. 
Now, let us fix $\delta=\bar{\delta}$ and let $\tau_j$ be as in Subsection \ref{sec2}. By Lemma \ref{lem2.8} there exists $\bar{\jmath}=\bar{\jmath}(\sigma) \in \mathbb{N}$ such that 
\begin{equation}\label{eeq13}
\left\|\tau_j u_{\bar{\delta}}-u_{\bar{\delta}}\right\|_{W^{s,G_{x,y}}(\mathbb{R}^{N})}<\frac{\sigma}{3}
\end{equation}
for $j$ large enough, say $j \geqslant \bar{\jmath}$.\\
For any $\varepsilon>0$ let us consider
$$
\rho_{\varepsilon}:=\tau_{\bar{\jmath}} u_{\bar{\delta}} * J_{\varepsilon},
$$
where $J_{\varepsilon}$ is the function defined in Subsection \ref{sec2}. Of course, $\rho_{\varepsilon} \in C^{\infty}\left(\mathbb{R}^N\right)$ by construction. Moreover, by the standard properties of the convolution (see e.g. \cite[Proposition~IV.18]{haim1983analyse}) we have that
\begin{equation}\label{eeq14}
\operatorname{Supp} \rho_{\varepsilon} \subseteq \operatorname{Supp} \tau_{\bar{\jmath}} u_{\bar{\delta}}+\overline{B}_{\varepsilon} .
\end{equation}
Also, by Lemma \ref{lem2.8} we have that
\begin{equation}\label{eeq15}
\operatorname{Supp} \tau_{\bar{\jmath}} u_{\bar{\delta}} \subseteq \overline{B}_{\bar{\jmath}+1} \cap \operatorname{Supp} u_{\bar{\delta}} .
\end{equation}
Now we claim that
\begin{equation}\label{eeq16}
\operatorname{Supp} \rho_{\varepsilon} \subseteq B_{\bar{\jmath}+2} \cap\left(\operatorname{Supp} u_{\bar{\delta}}+B_{2 \varepsilon}\right)
\end{equation}
if $\varepsilon$ is sufficiently small (possibly in dependence on $\sigma$ ). Indeed: let $P \in \operatorname{Supp} \rho_{\varepsilon}$. Then, by \eqref{eeq14}, there exists $Q \in \operatorname{Supp} \tau_{\bar{\jmath}} u_{\bar{\delta}}$ such that $|P-Q| \leqslant \varepsilon<2 \varepsilon$. Thus, by \eqref{eeq15}, we have $|Q| \leqslant \bar{\jmath}+1$ and $Q \in \operatorname{Supp} u_{\bar{\delta}}$. In particular, $|P| \leqslant|Q|+|P-Q| \leqslant$ $\bar{\jmath}+1+2 \varepsilon<\bar{\jmath}+2$ for $\varepsilon$ small enough, and this proves \eqref{eeq16}.\\
From \eqref{eeeq2} and \eqref{eeq16}, we deduce that Supp $\rho_{\varepsilon}$ is compact and contained in $\Omega$, as long as $\varepsilon$ is small enough, say $2 \varepsilon<a(N, u, \bar{\delta}, \bar{\jmath}+2, \Omega)$ in the notation of Lemma \ref{lemmad}. As a consequence of this, we have
$$
\rho_{\varepsilon} \in C_0^{\infty}(\Omega),
$$
for $\varepsilon$ small enough.\\
Furthermore, by Lemma \ref{lem2.06} there exists $\bar{\varepsilon}=\bar{\varepsilon}(\sigma)>0$ such that
\begin{equation}\label{eeq17}
\left\|\rho_{\varepsilon}-\tau_{\bar{\jmath}} u_{\bar{\delta}}\right\|_{W^{s,G_{x,y}}(\mathbb{R}^{N})}<\frac{\sigma}{3}
\end{equation}
for $\varepsilon$ small, say $\varepsilon \leqslant \bar{\varepsilon}$.\\
Hence, by $\eqref{eeq12},\eqref{eeq13}$ and \eqref{eeq17}, we have
$$
\begin{aligned}
\left\|u-\rho_{\varepsilon}\right\|_{W^{s,G_{x,y}}(\mathbb{R}^{N})} &\leqslant\left\|u-u_{\bar{\delta}}\right\|_{W^{s,G_{x,y}}(\mathbb{R}^{N})}
\\&+\left\|u_{\bar{\delta}}-\tau_{\bar{\jmath}} u_{\bar{\delta}}\right\|_{W^{s,G_{x,y}}(\mathbb{R}^{N})}+\left\|\tau_{\bar{\jmath}} u_{\bar{\delta}}-\rho_{\varepsilon}\right\|_{W^{s,G_{x,y}}(\mathbb{R}^{N})}\\
&<\frac{\sigma}{3}+\frac{\sigma}{3}+\frac{\sigma}{3}=\sigma .
\end{aligned}
$$
The arbitrariness of $\sigma$ concludes the proof of Theorem \ref{theo1}.
\end{proof}

\begin{proof}[Proof of Theorem~{\upshape\ref{theo2}}]
The proof of Theorem \ref{theo2} is similar to the one of Theorem 6 in \cite{fiscella2015density}, where the authors use an appropriate partition of unity in order to reduce the problem locally to the case of a hypograph and thus use Theorem \ref{theo1}.
\end{proof}
\begin{remark}\label{remark1} The sequence of function $\rho_{\varepsilon}$ in Theorem \ref{theo1} is supported in the vicinity of the support of the original function $u$. More precisely, fixed any $\gamma>0$ there exists $\varepsilon_\gamma>0$ such that for any $\varepsilon \in\left(0, \varepsilon_\gamma\right]$ one has that
$$
\operatorname{Supp} \rho_{\varepsilon} \subseteq \operatorname{Supp} u+B_\gamma
$$
Indeed, by construction we have
$$
\operatorname{Supp} u_{\bar{\delta}} \subseteq \operatorname{Supp} u+B_{2 \bar{\delta}} .
$$
This and \eqref{eeq16} yield that
$$
\operatorname{Supp} \rho_{\varepsilon} \subseteq \operatorname{Supp} u_{\bar{\delta}}+B_{2 \varepsilon} \subseteq \operatorname{Supp} u+B_{2 \bar{\delta}}+B_{2 \varepsilon} \subseteq \operatorname{Supp} u+B_{2(\varepsilon+\bar{\delta})},
$$
thus checking Remark \ref{remark1}.
\end{remark}

\bibliographystyle{plain}

\begin{thebibliography}{99}
\bibitem{Adams1975}ADAMS, R. A. Sobolev Spaces. Academic Press, New York, 1975.
\bibitem{abss2022class}AZROUL, E., BENKIRANE, A., SHIMI, M., AND SRATI, M. On a Class of Nonlocal
Problems in New Fractional Musielak-Sobolev Spaces. Appl. Anal. 101, 6 (2020),
1933–1952.
\bibitem{abss2023emb}AZROUL, E., BENKIRANE, A., SHIMI, M., AND SRATI, M. Embedding and Extension
Results in Fractional Musielak-Sobolev Spaces. Appl. Anal. 102, 1 (2021), 195–219.
\bibitem{baalal2018traces}BAALAL, A., AND BERGHOUT, M. Traces and Fractional Sobolev Extension Domains
with Variable Exponent. Int. J. Math. Anal. (N.S.) 12, 2 (2018), 85–98.
\bibitem{babm2018density}BAALAL, A., AND BERGHOUT, M. Density Properties for Fractional Sobolev Spaces
with Variable Exponents. Ann. Funct. Anal. 10, 3 (2019), 308–324.
\bibitem{baalal2024density} BAALAL, A., EL WAZNA, A., AND ZAOUI, M. A. Density Properties for Orlicz
Sobolev Spaces with Fractional Order. Rend. Circ. Mat. Palermo, II. Ser (2024), 1–16.

\bibitem{bamhoh2024onthefract} BAHROUNI, A., MISSAOUI, H., AND OUNAIES, H. On the Fractional
Musielak-Sobolev Spaces in $\mathbb{R}^{d}$: Embedding Results and Applications. J. Math.
Anal. Appl. 537, 1 (2024), Paper No. 128284, 32.

\bibitem{barv2018onanew}BAHROUNI, A., AND R\u{A}DULESCU , V. D. On a new Fractional Sobolev Space
and Applications to Nonlocal Variational Problems with Variable Exponent. Discrete
Contin. Dyn. Syst. Ser. S 11, 3 (2018), 379–389.

\bibitem{bsoh2020embedding}BAHROUNI, S., AND OUNAIES, H. Embedding Theorems in the Tractional
Orlicz-Sobolev Space and Applications to Non-local Problems. Discrete Contin. Dyn.
Syst. 40, 5 (2020), 2917–2944.

\bibitem{bot2020basic}BAHROUNI, S., OUNAIES, H., AND TAVARES, L. S. Basic Results of Fractional
Orlicz-Sobolev Space and Applications to Non-local Problems. Topol. Methods
Nonlinear Anal. 55, 2 (2020), 681–695.

\bibitem{bor2023anewclass}BOUJEMAA, H., OULGIHT, B., AND RAGUSA, M. A. A New Class of Fractional
Orlicz-Sobolev Space and Singular Elliptic Problems. J. Math. Anal. Appl. 526, 1
(2023), Paper No. 127342, 42.

\bibitem{haim1983analyse}BREZIS, H. Analyse Fonctionnelle : Théorie et Applications. Masson, Paris, 1983.

\bibitem{bms2019radial}BREZIS, H., MIRONESCU, P., AND SHAFRIR, I. Radial Extensions in Fractional
Sobolev Spaces. Rev. R. Acad. Cienc. Exactas Fís. Nat. Ser. A Mat. RACSAM 113, 2(2019), 707–714.

\bibitem{aacs2023onfrac} DE ALBUQUERQUE, J. C., DE ASSIS, L. R. S., CARVALHO, M. L. M., AND SALORT,
A. On Fractional Musielak-Sobolev Spaces and Applications to Nonlocal Problems. J.
Geom. Anal. 33, 4 (2023), Paper No. 130, 37.

\bibitem{dlrj2017traces}DEL PEZZO, L. M., AND ROSSI, J. D. Traces for Fractional Sobolev Spaces with
Variable Exponents. Adv. Oper. Theory 2, 4 (2017), 435–446.

\bibitem{di2012hitchhikers}DI NEZZA, E., PALATUCCI, G., AND VALDINOCI, E. Hitchhiker’s Guide to the
Fractional Sobolev Spaces. Bull. Sci. Math. 136, 5 (2012), 521–573.

\bibitem{BjSa2019}FERNÁNDEZ BONDER, J., AND SALORT, A. M. Fractional Order Orlicz-Sobolev
Spaces. J. Funct. Anal. 277, 2 (2019), 333–367.

\bibitem{fiscella2015density}FISCELLA, A., SERVADEI, R., AND VALDINOCI, E. Density Properties for Fractional Sobolev Spaces. Ann. Acad. Sci. Fenn. Math 40, 1 (2015), 235–253.

\bibitem{kjv2017fract}KAUFMANN, U., ROSSI, J. D., AND VIDAL, R. E. Fractional Sobolev Spaces with
Variable Exponents and Fractional p(x)-Laplacians. Electron. J. Qual. Theory Differ.
Equ., 76 (2017), 1–10.

\bibitem{km2023bourgan}KIM, M. Bourgain, Brezis and Mironescu Theorem for Fractional Sobolev Spaces with
Variable Exponents. Ann. Mat. Pura Appl. (4) 202, 6 (2023), 2653–2664.

 \bibitem{kovavcik1991spaces} KOV{\'A}{\v{C}}IK, O., AND R{\'A}KOSN{\'I}K, J. On Spaces $L^{p(x)}$ and $W^{k,p(x)}$. Czech. Math. J. 41, 4 (1991), 592–618.
  \bibitem{kajofs1977} KUFNER, A., JOHN, O., AND FU{\v{C}}{\'I}K, S. Function Spaces. Noordhoff, Leyden, 1977.
  
  \bibitem{mmrv2008neu}MIH\u{A}ILESCU, M., AND R\u{A}DULESCU, V. Neumann Problems Associated to
Nonhomogeneous Differential Operators in Orlicz-Sobolev Spaces. Ann. Inst. Fourier (Grenoble) 58, 6 (2008), 2087–2111.

\bibitem{musju1983Orlicz}MUSIELAK, J. Orlicz Spaces and Modular Spaces, vol. 1034 of Lecture Notes in
Mathematics. Springer-Verlag, Berlin, 1983.

\bibitem{vigelis2011musielak}VIGELIS, R. F. On Musielak-Orlicz Function Spaces and Applications to Information Geometry. PhD thesis, 2011. Thesis (Ph.D.)–Universidade Federal do Ceará.

\bibitem{ayya2020someapp} YOUSSFI, A., AND AHMIDA, Y. Some Approximation Results in Musielak-Orlicz
Spaces. Czechoslovak Math. J. 70(145), 2 (2020), 453–471.

\end{thebibliography}

% ------------------------------------------------------------------------
\end{document}